\documentclass[12pt]{amsart}
\usepackage{amssymb}
\usepackage{amsfonts}
\usepackage{amssymb,latexsym}
\usepackage{enumerate}
\usepackage{url} 
\usepackage{mathrsfs}
\makeatletter
\@namedef{subjclassname@2010}{%
  \textup{2010} Mathematics Subject Classification}
\makeatother

  [2002/01/19 v2.2g %
    AMS font definitions%
  ]
\DeclareFontFamily{U}{euf}{}
\DeclareFontShape{U}{euf}{m}{n}{%
  <5><6><7><8><9>gen*eufm%
  <10><10.95><12><14.4><17.28><20.74><24.88>eufm10%
  }{}
\DeclareFontShape{U}{euf}{b}{n}{%
  <5><6><7><8><9>gen*eufb%
  <10><10.95><12><14.4><17.28><20.74><24.88>eufb10%
  }{}

  [2002/01/19 v2.2g %
    AMS font definitions%
  ]
\DeclareFontFamily{U}{msb}{}
\DeclareFontShape{U}{msb}{m}{n}{%
  <5><6><7><8><9>gen*msbm%
  <10><10.95><12><14.4><17.28><20.74><24.88>msbm10%
  }{}

  [2002/01/19 v2.2g %
    AMS font definitions%
  ]
\DeclareFontFamily{U}{msa}{}
\DeclareFontShape{U}{msa}{m}{n}{%
  <5><6><7><8><9>gen*msam%
  <10><10.95><12><14.4><17.28><20.74><24.88>msam10%
  }{}

\newtheorem{theorem}{Theorem}[section]
\newtheorem{lemma}[theorem]{Lemma}
\newtheorem{proposition}[theorem]{Proposition}
\newtheorem{corollary}[theorem]{Corollary}

\theoremstyle{definition}

\theoremstyle{remark}
\newtheorem{remark}[theorem]{Remark}

\numberwithin{equation}{section} 

\begin{document}

\title[$p$-adic multiple Barnes-Euler zeta functions]{On $p$-adic multiple Barnes-Euler zeta functions and the corresponding Log Gamma functions}

\author{Su Hu}
\address{Department of Mathematics, South China University of Technology, Guangzhou, Guangdong 510640, China}
\email{mahusu@scut.edu.cn}

\author{Min-Soo Kim}
\address{Division of Mathematics, Science, and Computers, Kyungnam University,
7(Woryeong-dong) kyungnamdaehak-ro, Masanhappo-gu, Changwon-si, Gyeongsangnam-do 51767, Republic of Korea}
\email{mskim@kyungnam.ac.kr}



\begin{abstract}
Suppose that $\omega_1,\ldots,\omega_N$ are positive real numbers and $x$ is a complex number with positive real part.
The multiple Barnes-Euler zeta function $\zeta_{E,N}(s,x;\bar\omega)$ with parameter vector $\bar\omega=(\omega_1,\ldots,\omega_N)$ is
defined as a deformation of  the Barnes multiple  zeta function as follows
$$
\zeta_{E,N}(s,x;\bar\omega)=\sum_{t_1=0}^\infty\cdots\sum_{t_N=0}^\infty
\frac{(-1)^{t_1+\cdots+t_N}}{(x+\omega_1t_1+\cdots+\omega_Nt_N)^s}.
$$

In this paper, based on the fermionic $p$-adic integral, we define the $p$-adic analogue of multiple Barnes-Euler zeta function $\zeta_{E,N}(s,x;\bar\omega)$ which we denote by $\zeta_{p,E,N}(s,x;\bar\omega).$   We  prove several properties of  $\zeta_{p,E,N}(s,x; \bar\omega)$, including the convergent Laurent series
expansion, the distribution formula, the difference equation, the reflection functional equation and the derivative formula.
By computing the values of this kind of $p$-adic zeta function at nonpositive integers, we show that it interpolates the
higher order Euler polynomials $E_{N,n}(x;\bar\omega)$ $p$-adically.

Furthermore, we define the corresponding multiple $p$-adic Diamond-Euler Log Gamma function. We also show that the multiple $p$-adic Diamond-Euler Log Gamma function ${\rm Log}\, \Gamma_{\! D,E,N}(x;\bar\omega)$ has an integral representation by the  multiple fermionic $p$-adic integral,  and it  satisfies   the distribution formula, the difference equation, the reflection functional equation, the derivative formula  and also the Stirling's series expansions.

\end{abstract}

\maketitle
\def\ord{\text{ord}_p}
\def\ordt{\text{ord}_2}
\def\o{\omega}
\def\la{\langle}
\def\ra{\rangle}
\def\Log{{\rm Log}\, \Gamma_{\! D,E,N}}

\section{Introduction}

Throughout this paper, we  use the following notations.
\begin{equation*}
\begin{aligned}
\qquad \mathbb{C}  ~~~&- ~\textrm{the field of complex numbers}.\\
\mathbb{R}^{+}~~~&- ~\textrm{the set of positive real numbers}.\\
\qquad p  ~~~&- ~\textrm{an odd prime number}. \\
\qquad\mathbb{Z}_p  ~~~&- ~\textrm{the ring of $p$-adic integers}. \\
\qquad\mathbb{Q}_p~~~&- ~\textrm{the field of fractions of}~\mathbb Z_p.\\
\qquad\mathbb C_p ~~~&- ~\textrm{the completion of a fixed algebraic closure}~\overline{\mathbb Q}_p~ \textrm{of}~\mathbb Q.\\
\qquad v_p ~~~&- ~\textrm{the $p$-adic valuation of}~\mathbb
C_p~\textrm{normalized so that}~ |p|_p=p^{-v_p(p)}=p^{-1}
\end{aligned} \end{equation*}

The Hurwitz-type Euler zeta function is defined as follows
\begin{equation}~\label{EZ}
\zeta_{E}(s,x)=\sum_{n=0}^{\infty}\frac{(-1)^{n}}{(n+x)^{s}}
\end{equation}
for Re$(s)>0$ (see \cite{Ler} and \cite[Eq.~(1.6)]{SK16}), which is a deformation of the
Hurwitz zeta functions
\begin{equation}~\label{HZ}
\zeta(s,x)=\sum_{n=0}^{\infty}\frac{1}{(n+x)^{s}}
\end{equation}
for Re$(s)>1$. The Riemann zeta function is given by $\zeta(s)=\zeta(s,1),$ which is absolutely convergent for Re$(s)>1.$
When $x=1$, (\ref{EZ}) reduces to the Euler zeta functions (also called the Dirichlet eta functions)
\begin{equation}\label{Z}
\zeta_E(s)=\sum_{n=1}^{\infty}\frac{(-1)^{n-1}}{n^{s}}.\end{equation}
This function can be analytically continued to the complex plane without any pole. In fact, in \cite[Theorem 2.5]{Ti},
it is shown that the Euler zeta functions $\zeta_E(s)$ is summable (in the sense of Abel) to $(1-2^{1-s})\zeta(s)$
for all values of $s.$

During  the 1960s, Kubota and Leopoldt gave $p$-adic analogues of the Riemann zeta functions and the corresponding $L$-series (Dirichlet $L$-functions).
Their generalizations to  the $p$-adic Hurwitz series, named the $p$-adic Hurwitz zeta functions, have been systematically studied by Cohen \cite{Cohen} and Tangedal and Young~\cite{TY}. Recently, Young found that the $p$-adic Hurwitz zeta
function has a connection with his definition of $p$-adic Arakawa-Kaneko zeta function by an identity (see \cite[Corollary 1]{Yo15}).

The multiple Hurwitz zeta function in the complex field $$
\zeta_{N}(s,x;\omega_1,\ldots,\omega_N)=\sum_{t_1=0}^\infty\cdots\sum_{t_N=0}^\infty
\frac{1}{(x+\omega_1t_1+\cdots+\omega_Nt_N)^s},
$$ has a long history, it began with  Barnes's study on the theory of the multiple Gamma function $\Gamma_k(a).$ (See \cite[Sec. 2.1]{SC}). 
In 1977, Shintani~\cite{Shintani} evaluated the special values of certain $L$-functions attached to real quadratic number fields in terms of  Barnes's multiple Hurwitz zeta functions,
and the work of Barnes has also been applied  in the study of the determinants of the Laplacians around 1980s and 1990s (see \cite{Var} and \cite[Sec. 5.1, 5.2]{SC}). Recently, Tangedal and Young \cite{TY} considered
the $p$-adic analogue of Barnes multiple zeta functions  $\zeta_{p,N}(s,x;\omega_1,\ldots,\omega_N)$ and the corresponding  $p$-adic multiple log gamma functions $G_{p,N}(x;\omega_1,\ldots,\omega_N)$.

In \cite{TY2},  Tangedal and Young showed that, for any real quadratic field with prime $p$ splits completely in it, the $p$-adic multiple log gamma function $G_{p,2}(x;\omega_{1},\omega_{2})$ presents a representation for  the derivative at $s=0$ of the $p$-adic partial zeta function associated with any element in certain narrow ray class groups.

The main aim of this paper is to define the    $p$-adic analogue of multiple Barnes-Euler zeta functions 
$$
\zeta_{E,N}(s,x;\bar\omega)=\sum_{t_1=0}^\infty\cdots\sum_{t_N=0}^\infty
\frac{(-1)^{t_1+\cdots+t_N}}{(x+\omega_1t_1+\cdots+\omega_Nt_N)^s}.
$$
and study their properties. Our approaches are mainly based on the theory of fermionic $p$-adic integrals. So we recall the definition and applications of these integrals in the following subsection.

\subsection{The fermionic $p$-adic integral and its applications} 
 Let $UD(\mathbb Z_p)$ be the space of all uniformly (or strictly) differentiable $\mathbb C_p$-valued functions on $\mathbb Z_p$
(see \cite[\S11.1.2]{Cohen}).
The fermionic $p$-adic integral $I_{-1}(f)$  on $\mathbb Z_p$ of a function $f\in UD(\mathbb Z_p)$ is
defined by
\begin{equation}\label{-q-e2}
I_{-1}(f)=\int_{\mathbb Z_p}f(a)d\mu_{-1}(a)
=\lim_{N\rightarrow\infty}\sum_{a=0}^{p^N-1}f(a)(-1)^a.
\end{equation}
 In view of (\ref{-q-e2}), for any $f\in UD(\mathbb Z_p),$ we have
\begin{equation}\label{-q-e-diff}
\int_{\mathbb Z_p}f(t+1)d\mu_{-1}(t)+\int_{\mathbb Z_p}f(t)d\mu_{-1}(t)=2f(0)
\end{equation}
(see \cite[p.~782, (7)]{KT2}).

The fermionic $p$-adic integral (\ref{-q-e2})
were independently found as special cases by Katz \cite[p.~486]{Katz} (in Katz's notation, the $\mu^{(2)}$-measure), Shiratani and
Yamamoto \cite{Shi}, Osipov \cite{Osipov}, Lang~\cite{Lang} (in Lang's notation, the $E_{1,2}$-measure), T. Kim~\cite{KT2} from very different viewpoints.
And they have many applications in number theory. First, they are nice tools for studying many special numbers and polynomials, in particular, the Euler numbers and polynomials.

The Euler  polynomials $E_k(x),k\in\mathbb N_0=\mathbb N\cup\{0\},$ are defined by the generating function
\begin{equation}\label{Eu-pol}
\frac{2e^{xt}}{e^t+1}=\sum_{k=0}^\infty E_k(x)\frac{t^k}{k!},
\end{equation}
and the integers $E_{k}=2^{k}E_{k}\left({1}/{2}\right),k\in\mathbb N_0,$ are called Euler numbers.
For example, $E_0=1,E_2=-1,E_4=5,$ and $E_6=-61.$
The Euler numbers and polynomials (so called by Scherk in 1825) appear in Euler's famous book,
Insitutiones Calculi Differentials (1755, pp.487--491 and p.522).
Notice that the Euler numbers with odd subscripts vanish, that is, $E_{2m+1}=0$ for all $m\in\mathbb N_0.$

In 1875, Stern \cite{Stern} gave a  brief sketch of a proof of the following congruence of Euler numbers modulo powers of two:
\begin{equation} E_k\equiv E_l~(\textrm{mod}~~ 2^n)\iff k\equiv l ~(\textrm{mod}~~ 2^n),\end{equation}
for any $n\in\mathbb{Z}^{+}$, and $k,l$ be even.
In 1910, Frobenius  gave a detail for Stern's sketch. In 1979,  Ernvall \cite{Ernvall} argued that he could not understand Frobenius's proof and provided 
his own proof  using umbral calculus. In 2002, Wagstaff \cite{Wagstaff} presented an induction proof and in 2004, Sun \cite{Sun} gave a combinational proof.

The $p$-adic approach to the above formula seems  to be simple and general. It is based on the following Witt's formula of Euler numbers, which gives a $p$-adic integral representation of Euler numbers:
$$E_{n}=\int_{\mathbb{Z}_{p}}(2a+1)^{n} d\mu_{-1}(a).$$
In 2009, by using the above Witt's formula, Kim \cite[Theorem 2.11]{Kim-E} gave a brief proof of the above Stern's congruence.
In 2010, also using the theory of $p$-adic integration, Ma\"iga \cite[Theorem 6]{Maiga} proved a Kummer type congruence  for higher order Euler numbers:
$$E_{rp^{k}+s}^{(q)}\equiv E_{rp^{k-1}+s}^{(q)}~~(\textrm{mod}~p^{k}\mathbb{Z}_{p}),$$  where $r,k$ and $s$ are integers such that $(r,s)=1$, $k\geq 1$ and $s\geq 0$. 

The fermionic $p$-adic integral may also apply to many other special polynomials and numbers, for example, the Genocchi numbers.  In \cite{Simsek}, by using the fermionic $p$-adic integral $I_{-1}(f)$, Cangul, Ozden and Simsek defined generating functions of higher order $w$-Genocchi numbers $G_{n,w}^{(N)}$:
\begin{equation}\label{Can09}
\begin{aligned}
\int_{\mathbb Z_p^N}w^{x_1+\cdots+x_N}e^{t({x_1+\cdots+x_N})}d\mu_{-1}(x_1)\cdots d\mu_{-1}(x_N)
&=\sum_{n=0}^\infty G_{n,w}^{(N)}\frac{t^n}{n!} \\
&=2^N\left(\frac{t}{we^t+1}\right)^N
\end{aligned}
\end{equation}
(see \cite[(7)]{Simsek}), they also presented a Witt-type integral representations for these numbers:
\begin{equation}\label{Can-thm3}
\int_{\mathbb Z_p^N}(x_1+\cdots+x_N)^nw^{x_1+\cdots+x_N}d\mu_{-1}(x_1)\cdots d\mu_{-1}(x_N)
=\frac{G_{n,w}^{(N)}}{N!\binom{n+N}{N}}
\end{equation}
(see \cite[Theorem 3]{Simsek}).

In \cite{KS}, using the fermionic $p$-adic integral (\ref{-q-e2}),
we also defined $\zeta_{p,E}(s,x)$, the $p$-adic analogue of  Hurwitz-type Euler zeta functions (\ref{EZ}), which interpolates (\ref{EZ}) at nonpositive integers (see \cite[Theorem 3.8(2)]{KS}),
so called the $p$-adic Hurwitz-type Euler zeta functions.

We have also proved several properties of $\zeta_{p,E}(s,x)$, including the analyticity, the convergent Laurent series
expansion, the distribution formula, the difference equation, the reflection functional equation, the derivative formula and the $p$-adic Raabe formula.

The  $p$-adic Hurwitz-type Euler zeta function $\zeta_{p,E}(s,x)$ has been used by Hu, Kim and Kim to give a definition for the $p$-adic Arakawa-Kaneko-Hamahata zeta functions, which interpolate Hamahata's poly-Euler polynomials at non-positive integers
(see \cite[Definition 1.4]{SK17}).  And from the properties of $\zeta_{p,E}(s,x)$, they also obtained many identities on the $p$-adic Arakawa-Kaneko-Hamahata zeta functions,
including their derivative formula, difference equation and  reflection formula (see \cite[Theorems 2.1, 2.3 and 2.5]{SK17}). Recently, the $p$-adic  function $\zeta_{p,E}(s,x)$  becomes as a special case of Young's  definition of a $p$-adic analogue for the generalized Barnes zeta functions of order zero associated to the function $f(t) = r\textrm{Li}_{k}((1-e^{-t})/r)/(1-e^{-t}),$ where $\textrm{Li}_{k}(z)=\sum_{k=1}^{\infty}\frac{z^{m}}{m^{k}}$ (see \cite[(2.35)]{Yo17}).

In \cite[Sec.5]{KS}, using these zeta functions as building blocks, we also gave a definition for the corresponding $L$-functions $L_{p,E}(\chi,s)$, so called $p$-adic Euler $L$-functions. In \cite{SK16}, we showed that a case of Gross's refined Dedekind zeta  functions (see \cite[Sec.1]{Gross}, \cite[Sec.2]{SK16} and \cite[Sec.7]{Aoki}) may be represented by a product of Euler $L$-functions $L_{E}(\chi,s)$ (see \cite[Propositions 3.2 and 3.3]{SK16}), so the $p$-adic function $L_{p,E}(\chi,s)$ gives a growth formula for the $(S,\{2\})$-refined class numbers of the maximal real subfields of $p$-cyclotomic number fields (see \cite[Theorem 1.2 and Sec.4]{SK16}).

In \cite{KS13}, also using the fermionic $p$-adic integral $I_{-1}(f)$ (\ref{-q-e2}), we defined the corresponding $p$-adic Log Gamma functions,
named the $p$-adic Diamond-Euler Log Gamma functions. For example, in \cite[Sec. 6]{KS13}, using the $p$-adic Hurwitz-type Euler zeta functions, we found that the derivative of the $p$-adic Hurwitz-type Euler zeta functions $\zeta_{p,E}(s,x)$ at $s=0$ may be represented by the $p$-adic Diamond-Euler Log Gamma functions. This led us to connect the $p$-adic Hurwitz-type Euler zeta functions to the $(S,\{2\})$-version of the abelian rank one Stark conjecture (see \cite{Va}).

\subsection{Main works} 
Suppose that $\omega_1,\ldots,\omega_N$ are positive real numbers and $x$ is a complex number with positive real part.
The Barnes multiple zeta function $\zeta_{N}(s,x;\bar\omega)$ with parameter vector $\bar\omega=(\omega_1,\ldots,\omega_N)$ is defined
for Re$(s)>N$ by
\begin{equation}\label{B-M-def}
\zeta_{N}(s,x;\omega_1,\ldots,\omega_N)=\sum_{t_1=0}^\infty\cdots\sum_{t_N=0}^\infty
\frac{1}{(x+\omega_1t_1+\cdots+\omega_Nt_N)^s},
\end{equation}
which may also be written more concisely as
\begin{equation}\label{B-E-def-2}
\zeta_{N}(s,x;\bar\omega)=\sum_{\bar t\in\mathbb Z_0^N}^\infty
 (x+\bar\omega\cdot\bar t\,)^{-s}
\end{equation} (see \cite[(2.2)]{TY}).
When $N=1$ and $\omega_{1}=1$ this reduces to the Hurwitz zeta functions (\ref{HZ}).
As a function of $s$, $\zeta_{N}(s,x;\bar\omega)$ is analytic for Re$(s)>N$ and it has a meromorphic continuation to all of $\mathbb{C}$ with simple poles at $s = 1,\ldots,N$.
The Barnes multiple zeta function $\zeta_{N}(s,x;\bar\omega)$ also satisfies the difference equation (see \cite[(2.3)]{TY}), the derivative formula \cite[(2.5)]{TY}, and it also interpolates the higher order Bernoulli polynomials  $B_{N,n}(x;\bar\omega)$
at nonpositive integers (see \cite[(2.6)]{TY}).

Suppose that $\omega_1,\ldots,\omega_N\in\mathbb C_p^\times,$ and let $\Lambda$ denote the $\mathbb Z_p$-linear span of $\{\omega_1,\ldots,\omega_N\}.$
For $x\in\mathbb C_p\setminus\Lambda,$ Tangedal and Young~\cite[(3.1)]{TY} defined the $p$-adic analogue of Barnes multiple zeta function $\zeta_{N}(s,x;\bar\omega)$  with parameter vector $\bar\omega=(\omega_1,\ldots,\omega_N)$
by
the following equality
\begin{equation}\label{mul-def-zeta}
\begin{aligned} \zeta_{p,N}(s,x;\bar\omega)&=\frac{1}{\omega_{1}\cdots\omega_{N}(s-1)\cdots(s-N)}\\
&\quad\times\int_{\mathbb Z_p}\cdots\int_{\mathbb Z_p}
\frac{(x+\omega_1 t_1+\cdots+\omega_N t_N)^{N}}{\langle x+\omega_1 t_1+\cdots+\omega_N t_N\rangle^s}\,
dt_{1}\cdots dt_{N}.
\end{aligned}
\end{equation}
In which, the multiple Volkenborn  integrals  are all computable as iterated Volkenborn integrals.
Here the Volkenborn integral for a function $f$ which is strictly differentiable on $\mathbb{Z}_{p}$ is defined by
(comparing with (\ref{-q-e2}) above of the definition of fermionic $p$-adic integrals)
\begin{equation}\label{vol-int}
\int_{\mathbb Z_p}f(a)da=\lim_{N\to\infty}\frac1{p^N}\sum_{a=0}^{p^N-1}f(a)
\end{equation}
and this limit always exists when $f\in UD(\mathbb Z_p)$ (see \cite[p.\,264]{Ro} and \cite[\S55]{Sc}).
This integral was introduced by Volkenborn
\cite{Vo} and he also investigated many important properties of
$p$-adic valued functions defined on the $p$-adic domain (see
\cite{Vo,Vo1}). At the $k$-th iteration, with
$1\leq k\leq N,$ we integrate
$$\int_{\mathbb Z_p}F_k(t_k,t_{k+1},\ldots,t_N)dt_{k},$$
if $F_k(t_k,t_{k+1},\ldots,t_N)$ are strictly differentiable on $\mathbb{Z}_{p}$ for each fixed vector $(t_{k+1},\ldots,t_N)\in\mathbb Z_p^{N-k}.$

Tangedal and Young~\cite{TY} proved several properties of  $\zeta_{p,N}(s,x; \bar\omega)$, including the convergent Laurent series
expansion (see \cite[Theorem 4.1]{TY}), the distribution formula, the difference equation, the reflection functional equation and the derivative formula
(see \cite[Theorem 3.2]{TY}).
They also showed that $\zeta_{p,N}(s,x;\bar\omega)$ nicely interpolates the
higher order Bernoulli polynomials  $B_{N,n}(x;\bar\omega)$ $p$-adically at nonpositive integers (see \cite[Theorem 3.2(v)]{TY}).
Furthermore, they  \cite[(3.11)]{TY} defined the $p$-adic multiple log gamma function $G_{p,N}(x,\bar\omega)$ as the derivative of
$\zeta_{p,N}(s,x; \bar\omega)$ at $s=0$. They also showed that the $p$-adic multiple log gamma function $G_{p,N}(x;\bar\omega)$ has an integral representation
by the  multiple Volkenborn  integral (see \cite[(3.12)]{TY}),  and it  satisfies  the distribution formula, the difference equation, the reflection functional equation, the derivative formula (see \cite[Theorem 3.4]{TY}) and the Stirling's series expansions (see \cite[Theorem 4.2]{TY}).

Recently, using the reflection formula for $p$-adic multiple zeta functions defined in \cite{TY}, Young obtained  strong congruences for sums of the form $\sum_{n=0}^{N}B_{n}V_{n+1}$, where $B_{n}$ denotes the Bernoulli number and $V_{n}$ denotes a Lucas sequence of the second kind (see \cite{Yo16}).

Suppose that $\omega_1,\ldots,\omega_N$ are positive real numbers and $x$ is a complex number with positive real part.
The multiple Barnes-Euler zeta function $\zeta_{E,N}(s,x;\bar\omega)$ with parameter vector $\bar\omega=(\omega_1,\ldots,\omega_N)$ is
defined as a deformation of  the Barnes multiple zeta function $\zeta_{N}(s,x;\bar\omega)$ (\ref{B-M-def}) as follows
\begin{equation}\label{B-E-def}
\zeta_{E,N}(s,x;\omega_1,\ldots,\omega_N)=\sum_{t_1=0}^\infty\cdots\sum_{t_N=0}^\infty
\frac{(-1)^{t_1+\cdots+t_N}}{(x+\omega_1t_1+\cdots+\omega_Nt_N)^s}
\end{equation}
for Re$(s)>0$, which may also be written more concisely as
\begin{equation}\label{B-E-def-2}
\zeta_{E,N}(s,x;\bar\omega)=\sum_{\bar t\in\mathbb Z_0^N}^\infty
(-1)^{|\bar t |}(x+\bar\omega\cdot\bar t\,)^{-s}.
\end{equation}
Here $|\bar t|=t_1+\cdots+t_N.$
When $N=1$ and $\omega_{1}=1$ this reduces to the Hurwitz-type Euler zeta functions (\ref{EZ}).
When $N=0$ there is no parameter vector $\bar\omega$ but we may still regard the above equation as defining
the function $\zeta_{E,N}(s,x;-)=x^{-s}.$ As a function of $s$, $\zeta_{E,N}(s,x;\bar\omega)$ is analytic for Re$(s)>0$
and continued analytically to $s\in\mathbb C$ without any pole (see (\ref{B-E-def-3})).
The multiple Barnes-Euler zeta function $\zeta_{E,N}(s,x;\bar\omega)$ also satisfies the difference equation (see Lemma \ref{ele-prop}(1)),
the derivative formula (see Lemma \ref{ele-prop}(3)), and it also interpolates the higher order Euler polynomials $E_{N,n}(x;\bar\omega)$
at nonpositive integers (see Lemma \ref{val-ne}).

Suppose that $\omega_1,\ldots,\omega_N\in\mathbb C_p^\times,$ and let $\Lambda$ denote the $\mathbb Z_p$-linear span of $\{\omega_1,\ldots,\omega_N\}.$
In this paper, inspired by Tangedal and Young's work \cite{TY}, based on the fermionic $p$-adic integral (\ref{-q-e2}), we define the $p$-adic analogue of multiple Barnes-Euler zeta function $\zeta_{E,N}(s,x;\bar\omega),$  which we denote by $\zeta_{p,E,N}(s,x;\bar\omega)$ (see (\ref{mul-def-E-zeta})).
It is analytic  in certain areas for $s\in\mathbb{C}_{p}$  and $x\in\mathbb{C}_p\setminus\Lambda$ (see Theorem~\ref{analytic}).
We will prove several properties of  $\zeta_{p,E,N}(s,x; \bar\omega)$, including the convergent Laurent series
expansion (see Theorem \ref{t-expan-p=ga}), the distribution formula, the difference equation, the reflection functional equation and the derivative formula
(see Theorem \ref{e-zeta-thm1}).
We will also show that $\zeta_{p,E,N}(s,x;\bar\omega)$ interpolates the
higher order Euler polynomials  $E_{N,n}(x;\bar\omega)$ $p$-adically at nonpositive integers (see Theorem \ref{e-zeta-thm1}(5)). 

Furthermore, we will  define the multiple $p$-adic Diamond-Euler Log Gamma function $\Log(x,\bar\omega)$ as the derivative of $x\zeta_{p,E,N}(s,x; \bar\omega)/\langle x\rangle(s-1)$ at $s=0$
(see (\ref{mul-e-log})). As a function of $x$, it is  locally analytic   on $\mathbb C_p\setminus\Lambda$ (see Theorem~\ref{analytic-gamma}).
We will  also show  that the multiple $p$-adic Diamond-Euler Log Gamma function $\Log(x;\bar\omega)$ has an integral representation by the
multiple fermionic $p$-adic integral (see Proposition \ref{mul-e-log-int}),  so it extends the definition of the  $p$-adic Diamond-Euler Log Gamma function given
in \cite[(2.1)]{KS13}. And it satisfies  the distribution formula, the difference equation, the reflection functional equation, the derivative formula
(see Theorem \ref{e-gamma-thm1}) and also the Stirling's series expansions (see Theorem \ref{Stirling}).

\section{Multiple Barnes-Euler zeta functions and Euler polynomials of higher order}
In this section, we give a brief review of the properties of the multiple Barnes-Euler zeta functions
(the reader may refer to \cite{Choi} for an excellent modern treatment). The Euler polynomials of higher order
arise naturally in the study of these functions and in the first part of section 3 we give a $p$-adic realization of these polynomials
in terms of multiple fermionic $p$-adic integrals.

Let $\mathbb N_0$ denote the set of nonnegative integers, that is $\mathbb N_0=\mathbb N\cup\{0\}$
and $\mathbb R^+$ the set of positive real numbers.
Suppose that $\omega_1,\ldots,\omega_N$ are positive real numbers and $x$ is a complex number with positive real part.
The Euler polynomial $E_{N,n}(x;\bar\omega)$ of order $N$ and degree $n$ with parameter vector $\bar\omega
=(\omega_1,\ldots,\omega_N)$ is defined by
\begin{equation}\label{Eu-poly-def}
\frac{2^Ne^{xt}}{\prod_{j=1}^N\left(e^{\omega_jt}+1\right)}=\sum_{n=0}^\infty E_{N,n}(x;\bar\omega)\frac{t^n}{n!}.
\end{equation}
For $N=1$ and $\omega_1=1,$ they coincide with usual Euler polynomials $E_n(x),$ that is, $E_{1,n}(x;1)=E_n(x)$ for all $n\in\mathbb N_0.$
When $N=0$ there is no parameter vector $\bar\omega$ but we still use the generating function to define
\begin{equation}\label{eu-N=0}
E_{0,n}(x;-)=x^n.
\end{equation}
If all $\omega_i=1$ then $E_{n}^{(N)}(x)=E_{N,n}(x;1,\ldots,1)$ are the polynomials studied in \cite{Choi}, \cite{No} and \cite{Yo}.

As the classical Riemann zeta functions (\ref{Z}), the multiple Barnes-Euler zeta function $\zeta_{E,N}(s,x;\bar\omega)$ defined in (\ref{B-E-def}) may also be represented by the Mellin transform as follows
\begin{equation}\label{B-E-def-3}
\zeta_{E,N}(s,x;\bar\omega)=\frac1{\Gamma(s)}\int_0^\infty t^se^{-xt}\prod_{j=1}^N\left(1+e^{-\omega_jt}\right)^{-1}\frac{dt}{t},
\end{equation}
where Re$(s)>0,$
since if substituting the  following power series  into (\ref{B-E-def-3})
\begin{equation}\label{re-1}
\prod_{j=1}^N\left(1+e^{-\omega_jt}\right)^{-1}=\sum_{m_1=0}^\infty\cdots\sum_{m_N=0}^\infty(-1)^{m_1+\cdots+m_N}e^{-t(m_1\omega_1+\cdots+m_N\omega_N)}
\end{equation}
and noticing that
\begin{equation}\label{re-2}
\int_0^\infty t^s e^{-\omega t}\frac{dt}{t}=\omega^{-s}\Gamma(s),
\end{equation}
then we have (\ref{B-E-def}), which is used as a starting point of our work.
By (\ref{B-E-def-3}),  $\zeta_{E,N}(s,x;\bar\omega)$ can be continued analytically to $s\in\mathbb C$ without any pole.

\begin{lemma}\label{ele-prop}
The following identities hold:
\begin{enumerate}
\item $\zeta_{E,N}(s,x+\omega_N;\bar\omega)+\zeta_{E,N}(s,x;\bar\omega)=\zeta_{E,N-1}(s,x;\omega_1,\ldots,\omega_{N-1}).$
\item $\zeta_{E,N}(s,cx;c\bar\omega)=c^{-s}\zeta_{E,N}(s,x;\bar\omega)$ for all $c\in\mathbb R^+.$
\item $\frac{\partial^m}{\partial x^m}\zeta_{E,N}(s,x;\bar\omega)=(-1)^m(s)_m\zeta_{E,N}(s+m,x;\bar\omega),$
where $(s)_m$ is the Pochhammer symbol defined by $(s)_m=s(s+1)\cdots(s+m-1)$ if $m\in\mathbb N$ and 1 if $m=0.$
\end{enumerate}
\end{lemma}

Following Ruijsenaars (see \cite[(3.8)]{Ru}), in (\ref{re-2}), replace $s$ by $s+k-1,$ it clear that $\zeta_{E,N}(s,x;\bar\omega)$ satisfies the equation
\begin{equation}\label{int-re-1}
\begin{aligned}
\zeta_{E,N}(s,x;\bar\omega)&=\frac1{\Gamma(s)}\sum_{k=0}^M\frac{(-1)^k}{k!}E_{N,k}(0;\bar\omega)x^{1-s-k}\Gamma(s+k-1) \\
&\quad+\frac1{\Gamma(s)}\int_0^\infty t^se^{-xt}\left(\frac1{\prod_{j=1}^N\left(1+e^{-\omega_jt}\right)}
-\sum_{k=0}^M\frac{(-1)^k}{k!}E_{N,k}(0;\bar\omega)t^{k-1}\right)\frac{dt}{t}
\\
&=E_{N,0}(0;\bar\o)x^{-s+1}\frac{\Gamma(s-1)}{\Gamma(s)}-E_{N,1}(0;\bar\o)x^{-s}\frac{\Gamma(s)}{\Gamma(s)} \\
&\quad+\sum_{k=2}^M\frac{(-1)^k}{k!}E_{N,k}(0;\bar\omega)x^{1-s-k}\prod_{l=0}^{k-2}(s+l) \\
&\quad+\frac1{\Gamma(s)}\int_0^\infty t^se^{-xt}\left(\frac1{\prod_{j=1}^N\left(1+e^{-\omega_jt}\right)}
-\sum_{k=0}^M\frac{(-1)^k}{k!}E_{N,k}(0;\bar\omega)t^{k-1}\right)\frac{dt}{t}.
\end{aligned}
\end{equation}
We set
\begin{equation}\label{ru1}
\log \Gamma_{E,N}(x;\bar\omega)=\frac{\partial}{\partial s}\zeta_{E,N}(s,x;\bar\omega)\biggl|_{s=0}.
\end{equation}
Then using $\frac1{\Gamma(s)}=s+O(s^s)$ for $s\to0$ and using $E_{N,0}(0;\bar\o)=1,$
Barnes's asymptotic formula for the derivative of $\zeta_{E,N}(s,x;\bar\omega)$
at $s=0$ can be written as:

\begin{theorem}\label{int-re-3} We have
$$
\begin{aligned}
\log \Gamma_{E,N}(x;\bar\omega)&=x(\log x -1)+E_{N,1}(0;\bar\o)\log x  \\
&\quad+\sum_{k=2}^M\frac{(-1)^k}{k!}E_{N,k}(0;\bar\omega)x^{-k+1}(k-2)!+R_{N,M}(x),
\end{aligned}
$$
where
$$R_{N,M}(x)=\int_0^\infty e^{-xt}\left(\frac1{\prod_{j=1}^N\left(1+e^{-\omega_jt}\right)}
-\sum_{k=0}^M\frac{(-1)^k}{k!}E_{N,k}(0;\bar\omega)t^{k-1}\right)\frac{dt}{t}.$$
\end{theorem}

Using different methods as in \cite{Choi}, we obtain the following result which connections the special values of
$\zeta_{E,N}(s,x;\bar\omega)$ to the Euler polynomials of order $N$.

\begin{lemma}\label{val-ne}
The special value of $\zeta_{E,N}(s,x;\bar\omega)$ at a rational integer $k$ ($k\in\mathbb N_0$) is  given by equating the coefficients of powers
of $t$ in the identity
$$\sum_{k=0}^\infty \zeta_{E,N}(-k,x;\bar\omega)\frac{t^k}{k!}=\frac{e^{xt}}{\prod_{j=1}^N\left(e^{\omega_jt}+1\right)}.$$
In particular, we have
$$\zeta_{E,N}(-k,x;\bar\omega)=\frac1{2^N}E_{N,k}(x;\bar\omega).$$
\end{lemma}
\begin{proof}
Applying the procedure  used to get uniformly convergent in the wider sense of the multiple series, we obtain the following
generating function of $E_{N,n}(x;\bar\omega)$:
\begin{equation}\label{be-eu}
\frac{2^Ne^{xt}}{\prod_{j=1}^N\left(e^{\omega_jt}+1\right)}=\sum_{m_1=0}^\infty\cdots\sum_{m_N=0}^\infty(-1)^{m_1+\cdots+m_N}e^{(x+\omega_1m_1+\cdots+\omega_Nm_N)t}.
\end{equation}
Letting $s=-k~(k\in\mathbb{N}_{0})$ in ~(\ref{B-E-def}), we get
\begin{equation}\label{be-eu-zeta}
\zeta_{E,N}(-k,x;\bar\omega)=\sum_{t_1=0}^\infty\cdots\sum_{t_N=0}^\infty(-1)^{t_1+\cdots+t_N}(x+\omega_1t_1+\cdots+\omega_Nt_N)^k.
\end{equation}
Since
$\left(\frac{d}{dt}\right)^ke^{(x+m_1\omega_1+\cdots+m_N\omega_N)t}\left|_{t=0}\right.=(x+\omega_1m_1+\cdots+\omega_Nm_N)^k,$
it is clear from (\ref{Eu-poly-def}) and (\ref{be-eu-zeta}) that
\begin{equation}\label{zeta-val-ca}
\begin{aligned}
\zeta_{E,N}(-k,x;\bar\omega)&=\left(\frac{d}{dt}\right)^k\left(\sum_{k=0}^\infty\zeta_{E,N}(-k,x;\bar\omega)\frac{t^k}{k!}\right)\\&=\left(\frac{d}{dt}\right)^k\sum_{m_1=0}^\infty\cdots\sum_{m_N=0}^\infty(-1)^{m_1+\cdots+m_N}e^{(x+\omega_1m_1+\cdots+\omega_Nm_N)t}\biggl|_{t=0} \\
&=\frac1{2^N}\left(\frac{d}{dt}\right)^k\frac{2^Ne^{xt}}{\prod_{j=1}^N\left(e^{\omega_jt}+1\right)}\biggl|_{t=0} \\
&=\frac1{2^N}\left(\frac{d}{dt}\right)^k\sum_{n=0}^\infty E_{N,n}(x;\bar\omega)\frac{t^n}{n!}\biggl|_{t=0},
\end{aligned}
\end{equation}
so that
\begin{equation}
\zeta_{E,N}(-k,x;\bar\omega)=\frac1{2^N}E_{N,k}(x;\bar\omega).
\end{equation}
This is our assertion for the value of $\zeta_{E,N}(s,x;\bar\omega)$ at non-positive integers.
\end{proof}

\begin{lemma}\label{ele-eu-poly}
The following identities hold:
\begin{enumerate}
\item For an even integer $M,$ we have
$$\begin{aligned}
\sum_{j=0}^{M-1}&(-1)^jE_{K,n+1}(x+j\omega_{K+1};\omega_1,\ldots,\omega_K) \\
&=\frac{1}{2}(E_{K+1,n+1}(x;\omega_1,\ldots,\omega_{K+1})-E_{K+1,n+1}(x+M\omega_{K+1};\omega_1,\ldots,\omega_{K+1})).
\end{aligned}$$
\item For an odd integer $M,$ we have
$$\begin{aligned}
\sum_{j=0}^{M-1}&(-1)^jE_{K,n+1}(x+j\omega_{K+1};\omega_1,\ldots,\omega_K) \\
&=\frac{1}{2}(E_{K+1,n+1}(x;\omega_1,\ldots,\omega_{K+1})+E_{K+1,n+1}(x+M\omega_{K+1};\omega_1,\ldots,\omega_{K+1})).
\end{aligned}$$
\end{enumerate}
\end{lemma}
\begin{proof}
Suppose that $M$ is an odd integer.
To see Part (2), from (\ref{be-eu-zeta}) and Lemma \ref{val-ne}, note that
\begin{equation}\label{pf-e-e}
\begin{aligned}
&E_{K+1,n+1}(x+M\omega_{K+1};\omega_1,\ldots,\omega_{K+1}) \\
&=2^{K+1}\sum_{t_1=0}^\infty\cdots\sum_{t_{K+1}=0}^\infty(-1)^{t_1+\cdots+t_{K+1}}(x+\omega_1t_1+\cdots+\omega_{K+1}(t_{K+1}+M))^{n+1} \\
&=-2^{K+1}\sum_{t_1=0}^\infty\cdots\sum_{t_{K+1}=M}^\infty(-1)^{t_1+\cdots+t_{K+1}}(x+\omega_1t_1+\cdots+\omega_{K+1}t_{K+1})^{n+1}
\end{aligned}
\end{equation}
and
\begin{equation}\label{pf-e-e-2}
\begin{aligned}
&E_{K+1,n+1}(x;\omega_1,\ldots,\omega_{K+1}) \\
&=2^{K+1}\sum_{t_1=0}^\infty\cdots\sum_{t_{K+1}=0}^\infty(-1)^{t_1+\cdots+t_{K+1}}(x+\omega_1t_1+\cdots+\omega_{K+1}t_{K+1})^{n+1} \\
&=2^{K+1}\sum_{t_1=0}^\infty\cdots\sum_{t_{K+1}=M}^\infty(-1)^{t_1+\cdots+t_{K+1}}(x+\omega_1t_1+\cdots+\omega_{K+1}t_{K+1})^{n+1} \\
&\quad+2^{K+1}\sum_{t_1=0}^\infty\cdots\sum_{t_{K+1}=0}^{M-1}(-1)^{t_1+\cdots+t_{K+1}}(x+\omega_1t_1+\cdots+\omega_{K+1}t_{K+1})^{n+1}.
\end{aligned}
\end{equation}
On the other hand, we have
\begin{equation}\label{pf-e-e-3}
\begin{aligned}
&2^{K+1}\sum_{t_1=0}^\infty\cdots\sum_{t_{K+1}=0}^{M-1}(-1)^{t_1+\cdots+t_{K+1}}(x+\omega_1t_1+\cdots+\omega_{K+1}t_{K+1})^{n+1}\\
&=2\sum_{j=0}^{M-1}(-1)^jE_{K,n+1}(x+j\omega_{K+1};\omega_1,\ldots,\omega_{K}).
\end{aligned}
\end{equation}
Thus, Part (2) follows from (\ref{pf-e-e}), (\ref{pf-e-e-2}) and (\ref{pf-e-e-3}).
The calculation of Part (1) is similar.
\end{proof}

\section{$p$-Adic multiple Barnes-Euler zeta and multiple $p$-adic Diamond-Euler Log Gamma functions}

In this section, we define the $p$-adic analogue of the multiple Barnes-Euler zeta function $\zeta_{E,N}(s,x;\bar\omega)$
and the corresponding $p$-adic Log Gamma functions, and we also study their properties.

For a vector $\bar\o=(\o_1,\ldots,\o_N)\in\mathbb C_p^N$ we define its norm $\|\bar\o\|_p$ by
$$\|\bar\o\|_p=\max\{|\o_1|_p,\ldots,|\o_N|_p\}.$$
The multiple fermionic $p$-adic integrals considered here are defined as iterated integrals. At the $k$-th iteration, with
$1\leq k\leq N,$ we integrate
$$\int_{\mathbb Z_p}F_k(t_k,t_{k+1},\ldots,t_N)d\mu_{-1}(t_k),$$
give that $F_k(t_k,t_{k+1},\ldots,t_N)$ are continuous functions on $\mathbb{Z}_{p}$ for each fixed vector $(t_{k+1},\ldots,t_N)\in\mathbb Z_p^{N-k}.$
Under these conditions, we use the notation
\begin{equation}\label{mul-int-def}
\int_{\mathbb Z_p^N}f(\bar t\,)d\mu_{-1}(\bar t\,), \quad\text{where } \bar t =(t_1,\ldots,t_N),
\end{equation}
to denote the $N$-fold iterated fermionic $p$-adic integral
\begin{equation}\label{mul-int-def-2}
\begin{aligned}
\int_{\mathbb Z_p^N}f(\bar t\,)d\mu_{-1}(\bar t\,)=\lim_{l_1\to\infty}\cdots\lim_{l_N\to\infty }
\sum_{t_1=0}^{p^{l_1}-1}\cdots\sum_{t_N=0}^{p^{l_N}-1}f(t_1,\ldots,t_N)(-1)^{t_1+\cdots+t_N}.
\end{aligned}
\end{equation}

\begin{lemma}\label{lem-1}
For any $x\in\mathbb C_p,$ $\omega_i\in\mathbb C_p^\times \;(i=1,\ldots,N)$ and $1\leq k\leq N,$ we have
$$\begin{aligned}
\int_{\mathbb Z_p}E_{k-1,n}(x+\omega_kt_k&+\cdots+\omega_Nt_N;\omega_1,\ldots,\omega_{k-1})d\mu_{-1}(t_k) \\
&=E_{k,n}(x+\omega_{k+1}t_{k+1}+\cdots+\omega_Nt_N;\omega_1,\ldots,\omega_{k}),
\end{aligned}$$
where $E_{0,n}(x;-)=x^n$ and $\omega_{k+1}t_{k+1}+\cdots+\omega_Nt_N=0$ when $k=N.$
\end{lemma}
\begin{proof}
By (\ref{-q-e2}) and Lemma \ref{ele-eu-poly}(2), we have
$$\begin{aligned}
\int_{\mathbb Z_p}&E_{k-1,n}(x+\omega_kt_k+\cdots+\omega_Nt_N;\omega_1,\ldots,\omega_{k-1})d\mu_{-1}(t_k) \\
&=\lim_{l\to\infty}\sum_{t_k=0}^{p^{l}-1}(-1)^{t_k}E_{k-1,n}(x+\omega_{k+1}t_{k+1}+\cdots+\omega_Nt_N+\omega_kt_k;\omega_1,\ldots,\omega_{k-1}) \\
&=\lim_{l\to\infty}\frac12(E_{k,n}(x+\omega_{k+1}t_{k+1}+\cdots+\omega_Nt_N;\omega_1,\ldots,\omega_{k}) \\
&\qquad\qquad +E_{k,n}(x+\omega_{k+1}t_{k+1}+\cdots+\omega_Nt_N+p^l\omega_k;\omega_1,\ldots,\omega_{k})) \\
&=E_{k,n}(x+\omega_{k+1}t_{k+1}+\cdots+\omega_Nt_N;\omega_1,\ldots,\omega_{k})
\end{aligned}$$
as desired.
\end{proof}

\begin{corollary}\label{co-lem-1}
For any $x\in\mathbb C_p$ and $\omega_i\in\mathbb C_p^\times \;(i=1,\ldots,N),$ we have
$$\int_{\mathbb Z_p^N}(x+\bar\omega\cdot\bar t\,)^nd\mu_{-1}(\bar t\,)=E_{N,n}(x;\bar\omega),$$
where $n\in\mathbb N_0.$
\end{corollary}
\begin{proof}
If we set $k=1$ in Lemma \ref{lem-1}, we have
\begin{equation}\label{p-euler}
\begin{aligned}
\int_{\mathbb Z_p}(x+\omega_1t_1&+\cdots+\omega_N t_N)^nd\mu_{-1}(t_1) \\
&=\int_{\mathbb Z_p}E_{0,n}(x+\omega_1t_1+\cdots+\omega_N t_N;-)d\mu_{-1}(t_1) \\
&=E_{1,n}(x+\omega_2t_2+\cdots+\omega_N t_N;\omega_1)
\end{aligned}
\end{equation}
since $E_{0,n}(x;-)=x^n.$ Using (\ref{p-euler}), if we apply Lemma \ref{lem-1} inductively for $k=2,\ldots,N,$ we obtain
$$\begin{aligned}
\int_{\mathbb Z_p}\int_{\mathbb Z_p}(x+\omega_1t_1&+\cdots+\omega_N t_N)^nd\mu_{-1}(t_1)\mu_{-1}(t_2) \\
&=\int_{\mathbb Z_p}E_{1,n}(x+\omega_2t_2+\cdots+\omega_N t_N;\omega_1)d\mu_{-1}(t_2) \\
&=E_{2,n}(x+\omega_3t_3+\cdots+\omega_N t_N;\omega_1, \omega_2).
\end{aligned}$$
Hence the assertion is clear.
\end{proof}
The projection  function $\langle x\rangle$ for all $x\in
\mathbb{C}_p^{\times}$  defined by Kashio \cite{Kashio} and
Tangedal-Young \cite{TY} will play a key role in our definitions.  
We recall the definition of $\langle x\rangle$ in the following paragraph.

Fixing an embedding of $\overline{\mathbb{Q}}$ into $\mathbb{C}_{p}$.
$p^{\mathbb{Q}}$ denotes the image in $\mathbb{C}_{p}^{\times}$  of
the set of positive real rational powers of $p$ under this
embedding, $\mu$ denotes the group of roots of unity in
$\mathbb{C}_{p}^{\times}$ of order not divisible by $p$. For
$x\in\mathbb{C}_{p}$, $|x|_{p}=1$, there exists a unique elements
$\hat{x}\in\mu$ such that $|x-\hat{x}|_{p} < 1$ (called the
Teichm\"uller representative of $x$); it may be defined by
$\hat{x}=\lim_{n\to\infty}x^{p^{n!}}$. We extend this definition to
$x\in\mathbb{C}_{p}^{\times}$ by
\begin{equation}
\hat{x}:=(\widehat{x/p^{v_{p}(x)}}),
\end{equation}
that is, we define $\hat{x}=\hat{u}$ if $x=p^{r}u$ with $p^{r}\in
p^{\mathbb{Q}}$ and $|u|_{p}=1$, then we define the function
$\langle\cdot\rangle$ on $\mathbb{C}_{p}^{\times}$ by
\begin{equation}\label{p-brac}
\langle x\rangle=p^{-v_{p}(x)}x/\hat{x}.
\end{equation}
For $x\in\mathbb C_p$ with $|x|_p<1$ the Iwasawa logarithm function  is defined by the usual power series
\begin{equation}\label{p-log}
\log_px=-\sum_{n=1}^\infty\frac{(1-x)^n}{n}
\end{equation}
and is extended to a continuous function on $\mathbb C_p^\times$ by defining $\log_p x=\log_p\langle x\rangle.$

Suppose that $\omega_1,\ldots,\omega_N\in\mathbb C_p^\times,$ and let $\Lambda$ denote the $\mathbb Z_p$-linear span of $\{\omega_1,\ldots,\omega_N\}.$
For all $x\in\mathbb C_p\setminus\Lambda,$ we define the $p$-adic multiple Barnes-Euler zeta function $\zeta_{p,E,N}(s,x;\bar\omega)$ by
\begin{equation}\label{mul-def-E-zeta}
\zeta_{p,E,N}(s,x;\bar\omega)=\int_{\mathbb Z_p^N}
\langle x+\bar\omega\cdot \bar t\, \rangle^{1-s} d\mu_{-1}(\bar t\,).
\end{equation}
Here, the function with $N=0$ is understood to be `there is no parameter vector $\bar\omega$.'
This suggests that the integral
\begin{equation}\label{E-zeta0}
\zeta_{p,E,0}(s,x;-)=\int_{\mathbb Z_p^N}
\langle x \rangle^{1-s} d\mu_{-1}(t)=\langle x \rangle^{1-s}
\end{equation}
is derived easily from the relation $\int_{\mathbb Z_p}d\mu_{-1}(t)=1$ (see \cite[Propositin 2,1]{KS}).

By \cite[Proposition 2.1]{TY} and (\ref{mul-def-E-zeta}), we have the following result on the analytic  of $\zeta_{p,E,N}(s,x;\bar\omega)$ in certain areas for $s\in\mathbb{C}_{p}$  and $x\in\mathbb{C}_p\setminus\Lambda$.

\begin{proposition}\label{analytic}
For any choices of $\o_1,\ldots,\o_N,x\in\mathbb C_p^\times$ with $x\not\in\Lambda$ the function $\zeta_{p,E,N}(s,x;\bar\o)$ is a $C^\infty$ function
of s on $\mathbb Z_p,$ and is an analytic function of s on a disc of positive radius about $s = 0;$ on this disc it is
locally analytic as a function of $x$ and independent of the choice made to define the $\langle \cdot\rangle$ function. If
$\o_1,\ldots,\o_N,x$ are so chosen to lie in a finite extension $K$ of $\mathbb Q_p$ whose ramification index over $\mathbb Q_p$ is less than $p-1$ then
$\zeta_{p,E,N}(s,x;\bar\o)$ is analytic for $s\in\mathbb C_p$ such that $|s|_p <|\pi|_p^{-1}p^{-1/(p-1)},$
where $(\pi)$ is the maximal ideal of the ring of integers $O_K$ of $K.$
If $s\in \mathbb Z_p,$ the function $\zeta_{p,N}(s,x;\bar\o)$ is locally analytic as a function of $x$ on $\mathbb C_p\setminus\Lambda.$
\end{proposition}

\begin{theorem}\label{e-zeta-thm1}
The function $\zeta_{p,E,N}(s,x;\bar\omega)$ has the following properties:
\begin{enumerate}
\item
For all $x\in\mathbb C_p\setminus\Lambda$ the function $\zeta_{p,E,N}(s,x;\bar\omega)$ satisfies the difference equation
$$\zeta_{p,E,N}(s,x+\omega_N;\bar\omega)+\zeta_{p,E,N}(s,x;\bar\omega)=2\zeta_{p,E,N-1}(s,x;\omega_1,\ldots,\omega_{N-1}),$$
with the convention $\zeta_{p,E,0}(s,x;-)=\langle x\rangle^{1-s}$ (see (\ref{E-zeta0})) and $\bar\omega =(\omega_1,\ldots,\omega_N).$
\item
For all $c\in\mathbb C_p^\times$ and all $x\in\mathbb C_p\setminus\Lambda$ we have
$$\zeta_{p,E,N}(s,cx;c\bar\omega)=\la c\ra^{1-s}\zeta_{p,E,N}(s,x;\bar\omega).$$
\item For all $x\in\mathbb C_p\setminus\Lambda$ we have the reflection function equation
$$\zeta_{p,E,N}(s,|\bar\o|-x;\bar\omega)=\zeta_{p,E,N}(s,x;\bar\omega).$$
Here we will write $|\bar\o|=\omega_1+\cdots+\o_N.$
\item For all $x\in\mathbb C_p\setminus\Lambda$ and all positive odd integers $m$ we have the multiplication formula (distribution relation)
$$\zeta_{p,E,N}(s,x;\bar\omega)=\la m\ra^{1-s}
\sum_{0\leq j_i<m}(-1)^{|\bar j|}\zeta_{p,E,N}\left(s,\frac{x+\bar j\cdot\bar\o}{m};\bar\omega\right),$$
where the sum is over all vectors $\bar j=(j_1,\ldots,j_N)$ with $0\leq j_i<m.$
In particular for any $k\in\mathbb N$ we have
$$\zeta_{p,E,N}(s,x;\bar\omega)=\sum_{0\leq j_i<p^k}(-1)^{|\bar j|}\zeta_{p,E,N}\left(s,\frac{x+\bar j\cdot\bar\o}{p^k};\bar\omega\right).$$
\item Suppose that $\o_1,\ldots,\o_N\in\bar{\mathbb Q}$ are positive real numbers and $x\in\bar{\mathbb Q}$ is a complex number with positive
real part. Under our fixed embedding of $\bar{\mathbb Q}$ into $\mathbb C_p,$ suppose that $|x|_p>\|\bar\o\|_p.$ Then for all $k\in\mathbb N,$ we have
$$\zeta_{p,E,N}(1-k,x;\bar\omega)=2^N\left(\frac{\la x\ra}{x}\right)^k \zeta_{E,N}(-k,x;\bar\omega).$$
\item
Suppose $\o_1,\ldots,\o_N,x\in\mathbb C_p^\times$ and $|x|_p>\|\bar\o\|_p.$ Then for any $m\in\mathbb N_0,$ the identity
$$\frac{\partial^m}{\partial x^m}\zeta_{p,E,N}(s,x;\bar\omega)=(-1)^m \left(\frac{\la x\ra}{x}\right)^m
(s-1)_m\zeta_{p,E,N}(s+m,x;\bar\omega)$$
holds if $s\in\mathbb Z_p;$ this identity also holds for $|s|_p<|\pi|_p^{-1}p^{-1/(p-1)}$ if $x$ and all $\o_i$  lie in a finite extension $K$
of $\mathbb Q_p$ whose ring of integers has maximal ideal $(\pi)$ with ramification index over $\mathbb Q_p$ less than $p-1.$
\end{enumerate}
\end{theorem}
\begin{proof}
(1) Following Tangedal and Young (see \cite[Theorem 3.2]{TY}), we define a function $f(t_N)$ on $\mathbb Z_p$ by the $(N-1)$-fold
iterated fermionic $p$-adic integral as follows
\begin{equation}\label{fN=0}
f(t_N)=\int_{\mathbb Z_p}\cdots\int_{\mathbb Z_p}
\frac1{\langle x+\omega_1 t_1+\cdots+\omega_N t_N\rangle^{s-1}} d\mu_{-1}(t_1)\cdots d\mu_{-1}(t_{N-1}).
\end{equation}
From (\ref{-q-e-diff}), we have
\begin{equation}\label{fN diff}
\int_{\mathbb Z_p}f(t_N+1)d\mu_{-1}(t_{N})+\int_{\mathbb Z_p}f(t_N)d\mu_{-1}(t_{N})=2f(0).
\end{equation}
If we put $t_N=0$ in (\ref{fN=0}), use (\ref{mul-def-E-zeta}) and (\ref{fN diff}), we get
$$\begin{aligned}
\zeta_{p,E,N}&(s,x+\omega_N;\bar\omega)+\zeta_{p,E,N}(s,x;\bar\omega) \\
&=\int_{\mathbb Z_p}f(t_N+1)d\mu_{-1}(t_{N})+\int_{\mathbb Z_p}f(t_N)d\mu_{-1}(t_{N}) \\
&=2f(0)  \\
&=2\int_{\mathbb Z_p}\cdots\int_{\mathbb Z_p}
\frac1{\langle x+\omega_1 t_1+\cdots+\omega_{N-1} t_{N-1}\rangle^{s-1}} d\mu_{-1}(t_1)\cdots d\mu_{-1}(t_{N-1}) \\
&=2\zeta_{p,E,N-1}(s,x;\omega_1,\ldots,\omega_{N-1}),
\end{aligned}$$
which prove (1).

(2) From the fact that $\la \cdot\ra$ is a multiplicative homomorphism, we have
$$\begin{aligned}
\zeta_{p,E,N}(s,cx;c\bar\omega)
&=\int_{\mathbb Z_p^N}\frac1{\langle c(x+\bar\omega\cdot\bar t\rangle^{s-1}} d\mu_{-1}(\bar t\,) \\
&=\frac1{\la c\ra^{s-1}}\zeta_{p,E,N}(s,x;\bar\omega).
\end{aligned}$$
Then the assertion is clear.

(3) We define
\begin{equation}\label{def-ft}
f(x,\bar t\,)=\frac1{\langle x+\bar\omega\cdot\bar t\rangle^{s-1}}.
\end{equation}
Since $$\int_{\mathbb Z_p}f(t+1)d\mu_{-1}(t)=\int_{\mathbb Z_p}f(-t)d\mu_{-1}(t)$$ (see \cite[Theorem 2.2(2)]{KS}),
we have
$$\begin{aligned}
\zeta_{p,E,N}&(s,|\bar\o|-x;\bar\omega)  \\
&=\int_{\mathbb Z_p^N}f(\omega_1+\cdots+\o_N-x,\bar t\,)d\mu_{-1}(\bar t\,) \\
&=\int_{\mathbb Z_p^N}\frac1{\langle \omega_1+\cdots+\o_N-x+\omega_1 t_1+\cdots+\omega_{N} t_{N}\rangle^{s-1}} d\mu_{-1}(\bar t\,) \\
&=\int_{\mathbb Z_p^N}\frac1{\langle \omega_1+\cdots+\o_{N-1}-x+\omega_1 t_1+\cdots+\omega_{N} (t_{N}+1)\rangle^{s-1}} d\mu_{-1}(\bar t\,) \\
&=\int_{\mathbb Z_p^N}f(\omega_1+\cdots+\o_{N-1}-x,t_1,\ldots,t_{N-1},t_N+1)d\mu_{-1}(\bar t\,) \\
&=\int_{\mathbb Z_p^N}f(\omega_1+\cdots+\o_{N-1}-x,t_1,\ldots,t_{N-1},-t_N)d\mu_{-1}(\bar t\,) \\
&=\cdots \\
&=\int_{\mathbb Z_p^N}f(-x,-t_1,\ldots,-t_N)d\mu_{-1}(\bar t\,) \\
&=\int_{\mathbb Z_p^N}f(x,\bar t\,)d\mu_{-1}(\bar t\,)=\zeta_{p,E,N}(s,x;\bar\omega).
\end{aligned}$$
The validity of the  penultimate equality depends on the fact that $\la-z\ra^s=\la z\ra^s;$ when $p\neq2$ this is valid
since $\la-z\ra=\la z\ra.$ For $p=2$ we have $\la-z\ra=-\la z\ra$ so the equation is not valid for general $s,$ but there is a disc about
$s=0$ on which $\la-z\ra^s=\exp_p(s\log_p(-\la z\ra))=\exp_p(s\log\la z\ra)=\la z\ra^s,$ so on this
disc the equation is valid (see \cite[p.~1250]{TY}).

(4) We note  that \cite[Theorem 2.2(3)]{KS}
\begin{equation}\label{mul-ft-4}
\int_{\mathbb Z_p}f(t)d\mu_{-1}(t)=\sum_{j=0}^{m-1}(-1)^j\int_{\mathbb Z_p}f(j+mt)d\mu_{-1}(t),
\end{equation}
where $m$ is an odd integer.
We may use (\ref{def-ft}) and (\ref{mul-ft-4}) to compute
$$\begin{aligned}
\zeta_{p,E,N}(s,x;\bar\omega)&=\int_{\mathbb Z_p^N}f(x,\bar t\,)d\mu_{-1}(\bar t\,) \\
&=\int_{\mathbb Z_p^{N-1}}\left(\int_{\mathbb Z_p}f(x,t_1,\ldots,t_N)d\mu_{-1}(t_N)\right)\cdots d\mu_{-1}(t_1) \\
&=\sum_{j_N=0}^{m-1}(-1)^{j_N}\int_{\mathbb Z_p^{N}}f(x,t_1,\ldots,j_N+mt_N)d\mu_{-1}(t_N)\cdots d\mu_{-1}(t_1) \\
&=\cdots \\
&=\sum_{0\leq j_i<m}(-1)^{|\bar j|} \int_{\mathbb Z_p^{N}}f(x,\bar j+m\bar t\,)d\mu_{-1}(\bar t\,)  \\
&=\sum_{0\leq j_i<m}(-1)^{|\bar j|} \int_{\mathbb Z_p^{N}}f(x+\bar j\cdot\bar\o,m\bar t\,)d\mu_{-1}(\bar t\,)  \\
&=\la m\ra^{1-s}\sum_{0\leq j_i<m}(-1)^{|\bar j|}
\int_{\mathbb Z_p^{N}}f\left(\frac{x+\bar j\cdot\bar\o}{m},\bar t\,\right)d\mu_{-1}(\bar t\,) \\
&=\la m\ra^{1-s}\sum_{0\leq j_i<m}(-1)^{|\bar j|}
\zeta_{p,E,N}\left(s,\frac{x+\bar j\cdot\bar\o}{m};\bar \o\right),
\end{aligned}$$
proving (4).

(5) Note that (\cite[p.~1251]{TY}) if $|x|_p>\|\bar\o\|_p,$ then
\begin{equation}\label{di-1}
\la x+\bar\o\cdot\bar t\,\ra=\frac{\la x\ra}{x}(x+\bar\o\cdot\bar t\,)\quad\text{for all }\bar t\in\mathbb Z_p^N.
\end{equation}
For given $x$ and $\bar\o$ satisfying these hypotheses, we have
\begin{equation}\label{di-2}
\frac{\partial}{\partial x}\la x+\bar\o\cdot\bar t\,\ra^s=s\la x+\bar\o\cdot\bar t\,\ra^s (x+\bar\o\cdot\bar t\,)^{-1}
\quad\text{uniformly for }\bar t\in\mathbb Z_p^N.
\end{equation}
Now by (\ref{mul-def-E-zeta}), (\ref{di-1}), Corollary \ref{co-lem-1} and Lemma \ref{val-ne}, we have
$$\begin{aligned}
\zeta_{p,E,N}(1-k,x;\bar\omega)
&=\int_{\mathbb Z_p^N} \la x+\bar\o\cdot\bar t\,\ra^{k}d\mu_{-1}(\bar t\,) \\
&=\left(\frac{\la x\ra}{x}\right)^k\int_{\mathbb Z_p^N}  (x+\bar\o\cdot\bar t\,)^{k}d\mu_{-1}(\bar t\,) \\
&=\left(\frac{\la x\ra}{x}\right)^k\int_{\mathbb Z_p^N} E_{0,k}(x+\bar\o\cdot\bar t;-) d\mu_{-1}(\bar t\,) \\
&=\left(\frac{\la x\ra}{x}\right)^k E_{N,k}(x;\bar\o) \\
&=2^N\left(\frac{\la x\ra}{x}\right)^k \zeta_{E,N}(-k,x;\bar\omega),
\end{aligned}$$
which proves (5).

(6) From (\ref{di-1}) and (\ref{di-2}) we have
$$\begin{aligned}
\frac{\partial}{\partial x}\zeta_{p,E,N}(s,x;\bar\omega)
&=\int_{\mathbb Z_p^N} \frac{\partial}{\partial x}\la x+\bar\o\cdot\bar t\,\ra^{1-s}d\mu_{-1}(\bar t\,) \\
&=(1-s)\int_{\mathbb Z_p^N} \la x+\bar\o\cdot\bar t\,\ra^{1-s}\frac{1}{x+\bar\o\cdot\bar t}d\mu_{-1}(\bar t\,) \\
&=(1-s)\frac{\la x\ra}{x}\int_{\mathbb Z_p^N} \la x+\bar\o\cdot\bar t\,\ra^{-s}d\mu_{-1}(\bar t\,)   \\
&=(1-s)\frac{\la x\ra}{x}\zeta_{p,E,N}(s+1,x;\bar\omega),
\end{aligned}$$
which proves (6) for $m=1.$ The general statement of (6) follows by induction on $m.$
\end{proof}

Given $\o_1,\ldots,\o_N\in\mathbb C_p^\times$ and $x\in\mathbb C_p\setminus\Lambda,$
we define the multiple $p$-adic Diamond-Euler Log Gamma function $\Log(x;\bar\o)$ by
\begin{equation}\label{mul-e-log}
\Log(x;\bar\omega)=\frac{x}{\la x\ra}\frac{\partial}{\partial s}\left(\frac1{s-1}\zeta_{p,E,N}(s,x;\bar\omega)\right)\biggl|_{s=0}.
\end{equation}

When $N=1$ and $\o_1=1$ this is the $p$-adic Diamond-Euler Log Gamma function
$${\rm Log}\, \Gamma_{\! D,E,1}(x;1)={\rm Log}\, \Gamma_{\! D,E}(x),\quad x\in \mathbb C_p\setminus\mathbb Z_p$$
(see \cite[p.~4235, (2.1)]{KS13}).

By Proposition \ref{analytic} and (\ref{mul-e-log}), we have the following result on the analytic  of $\Log(x;\bar\omega)$ as a function of $x$.

\begin{theorem}\label{analytic-gamma}
For any $\o_1,\ldots,\o_N\in\mathbb C_p^\times$ the function $\Log(x;\bar\omega)$ is independent of the choice made to define the
$\la \cdot\ra$ function, and is locally analytic as a function of $x\in\mathbb C_p\setminus\Lambda$.
\end{theorem}

In what follows, we shall give some properties of multiple $p$-adic Diamond-Euler Log Gamma function $\Log(x;\bar\omega)$.

Using the Volkenborn integral, Diamond \cite{Di,Di2} gave a definition for the $p$-adic Log Gamma functions.
Following Diamond, using the multiple fermionic $p$-adic integral on $\mathbb Z_p,$
we obtain an integral representation of $\Log(x;\bar\omega),$ which shows that it is indeed a generalization  of  the $p$-adic Diamond-Euler Log Gamma function defined in \cite[(2.1)]{KS13}, as follows:

\begin{proposition} \label{mul-e-log-int}
For $\o_1,\ldots,\o_N\in\mathbb C_p^\times$ and $x\in\mathbb C_p\setminus\Lambda$ we have
$$\Log(x;\bar\omega)=\int_{\mathbb Z_p^N} ((x+\bar\o\cdot\bar t\,)\log_p(x+\bar\o\cdot\bar t\,)-(x+\bar\o\cdot\bar t\,))d\mu_{-1}(\bar t\,),$$
where $\log_p$ is the Iwasawa $p$-adic logarthm.
\end{proposition}
\begin{proof}
Note that
$\frac{\partial}{\partial s}\la x+\bar\o\cdot\bar t\,\ra^{1-s}=-\la x+\bar\o\cdot\bar t\,\ra^{1-s}\log_p\la x+\bar\o\cdot\bar t\,\ra.$
From (\ref{di-1}) and (\ref{mul-e-log}) the assertion is clear, since $\log_p x=\log_p\la x\ra$ on $\mathbb C_p^\times$ (see \cite[p.~1244, (2.18)]{TY}).
\end{proof}

\begin{remark}
Imai \cite[Proposition 1]{Im} also defined a $p$-adic log multiple $\Gamma$ function and formulated a relationship with a special value of a $p$-adic analogue of
the multiple $\zeta$-function, not with the derivatives.
\end{remark}

We put
\begin{equation}\label{deri}
D=\frac{\partial}{\partial x}.
\end{equation}
Furthermore, if $D^k$ denotes the $n$th derivative, we define
\begin{equation}\label{der-h}
\psi_{p,E,N}^{(k)}(x;\bar\omega)=D^k\Log(x;\bar\omega)
\end{equation}
for $x\in\mathbb C_p\setminus\Lambda.$ We also denote
$\psi_{p,E,N}^{(1)}(x;\bar\omega)$ by $\psi_{p,E,N}(x;\bar\omega).$

\begin{theorem}\label{e-gamma-thm1}
The function $\Log(x;\bar\omega)$ has the following properties:
\begin{enumerate}
\item
For all $x\in\mathbb C_p\setminus\Lambda$ the function $\Log(x;\bar\omega)$ satisfies the difference equation
$$\begin{aligned}
\Log(x+\omega_N;\bar\omega)+&\Log(x;\bar\omega) \\
&=2{\rm Log}\, \Gamma_{\! D,E,N-1}(x;\omega_1,\ldots,\omega_{N-1}),
\end{aligned}$$
with the convention ${\rm Log}\, \Gamma_{\! D,E,0}(x;-)=x(\log_p x-1).$
\item
For all $c\in\mathbb C_p^\times$ and all $\in\mathbb C_p\setminus\Lambda$ we have
$$\Log(cx;c\bar\omega)=c\left(\Log(x;\bar\omega)+\frac x{\la x\ra}\zeta_{p,E,N}(0,x;\bar\o)\log_p c\right).$$
\item For all $x\in\mathbb C_p\setminus\Lambda$ and $|x|_p>\|\bar\o\|_p$ we have the reflection function equation
$$\Log(|\bar\o|-x;\bar\omega)+\Log(x;\bar\omega)=0,$$
where $|\bar\o|=\omega_1+\cdots+\o_N.$
\item For all $x\in\mathbb C_p\setminus\Lambda$ and all positive odd integers $m$ we have the multiplication formula (distribution relation)
$$\begin{aligned}
\Log(x;\bar\omega)&= m \sum_{0\leq j_i<m}(-1)^{|\bar j|}\Log\left(\frac{x+\bar j\cdot\bar\o}{m};\bar\omega\right) \\
&\quad+E_{N,1}(x;\bar\o)\log_p m,
\end{aligned}$$
where the sum is over all vectors $\bar j=(j_1,\ldots,j_N)$ with $0\leq j_i<m.$
In particular for any $k\in\mathbb N$ we have
$$\Log(x;\bar\omega)= p^k \sum_{0\leq j_i< p^k}(-1)^{|\bar j|}\Log\left(\frac{x+\bar j\cdot\bar\o}{p^k};\bar\omega\right).$$

\item For all $x\in\mathbb C_p\setminus\Lambda$ and for $|x|_p>\|\bar\o\|_p$ we have
$$\psi_{p,E,N}^{(k+1)}(x;\bar\omega)
=(-1)^{k+1}\int_{\mathbb Z_p^N} \frac1{(x+\bar\o\cdot\bar t\,)^k}d\mu_{-1}(\bar t\,)$$
and we may write this as
$$\psi_{p,E,N}^{(k+1)}(x;\bar\omega)=(-1)^{k+1}\left(\frac {\la x\ra}x\right)^{k}\zeta_{p,E,N}(k+1,x;\bar\omega),$$
where $k\in\mathbb N.$ In particular
$\psi_{p,E,N}(x;\bar\omega)=\int_{\mathbb Z_p^N} \log_p(x+\bar\o\cdot\bar t\,) d\mu_{-1}(\bar t\,).$
\end{enumerate}
\end{theorem}
\begin{proof}
(1) Suppose that $N=0.$ By (\ref{E-zeta0}) and (\ref{mul-e-log}), we get
$${\rm Log}\, \Gamma_{\! D,E,0}(x;-)=\frac{x}{\la x\ra}\frac{\partial}{\partial s}\left(\frac1{s-1}\la x\ra^{1-s}\right)\biggl|_{s=0}
=x\log_p x -x.$$
Therefore Part (1) is obtained from Theorem \ref{e-zeta-thm1}(1) by using (\ref{mul-e-log}).

(2) From Theorem \ref{e-zeta-thm1}(2) we have
\begin{equation}\label{zeta-g-1}
\begin{aligned}
\frac{cx}{\la cx\ra}\frac{\partial}{\partial s}&\left(\frac1{s-1}\zeta_{p,E,N}(s,cx;c\bar\omega)\right)\biggl|_{s=0} \\
&=\frac{cx}{\la cx\ra}\frac{\partial}{\partial s}\left(\la c\ra^{1-s}\frac1{s-1}\zeta_{p,E,N}(s,x;\bar\omega) \right)\biggl|_{s=0}.
\end{aligned}
\end{equation}
Therefore by (\ref{mul-e-log}) and (\ref{zeta-g-1}) we obtain
$$
\begin{aligned}
\Log(cx;c\bar\omega)
&=\frac{cx}{\la cx\ra}\frac{\partial}{\partial s}\left(\la c\ra^{1-s} \right)\frac1{s-1}\zeta_{p,E,N}(s,x;\bar\omega)\biggl|_{s=0} \\
&\quad+\frac{c}{\la c\ra}\la c\ra^{1-s}\frac{x}{\la x\ra}\frac{\partial}{\partial s}\left(\frac1{s-1}\zeta_{p,E,N}(s,x;\bar\omega) \right)\biggl|_{s=0} \\
&=c\frac{x}{\la x\ra} \zeta_{p,E,N}(0,x;\bar\omega)\log_p c + c\Log(x;\bar\omega),
\end{aligned}
$$
which proves (2). In this calculation we have used the fact that $\la \cdot\ra$ is a multiplicative homomorphism.

(3) We will write $|\bar\o|=\omega_1+\cdots+\o_N.$ From Theorem \ref{e-zeta-thm1}(3) we have
\begin{equation}\label{zeta-g-3}
\begin{aligned}
\frac{|\bar\o|-x}{\la |\bar\o|-x\ra}&\frac{\partial}{\partial s}\left(\frac1{s-1}\zeta_{p,E,N}(s,|\bar\o|-x;\bar\omega)\right)\biggl|_{s=0} \\
&=\frac{|\bar\o|-x}{\la |\bar\o|-x\ra}\frac{\partial}{\partial s}\left(\frac1{s-1}\zeta_{p,E,N}(s,x;\bar\omega) \right)\biggl|_{s=0}.
\end{aligned}
\end{equation}
Since $\la -z\ra=\la z\ra,$ we see easily that for $|x|_p>\|\bar\o\|_p$
$$\frac{|\bar\o|-x}{\la |\bar\o|-x\ra}\frac{\la x\ra}{x}=\frac{|\bar\o|-x}{\frac{\la x\ra}{x}(x-|\bar\o|)}\frac{\la x\ra}{x}=-1.$$
Hence we obtain, by (\ref{mul-e-log}) and (\ref{zeta-g-3}), that
$$\Log(|\bar\o|-x;\bar\omega)=-\Log(x;\bar\omega),$$
which proves (3).

(4) Since $\la \cdot\ra$ is a multiplicative homomorphism, we can obtain the following identity
\begin{equation}\label{note}
\frac{\frac{x+\bar j\cdot\bar\o}{m}}{\left\la \frac{x+\bar j\cdot\bar\o}{m}\right\ra}=\frac{\la m\ra}m\frac{x}{\la x\ra}.
\end{equation}
Let $m$ be an odd integer. Then by Lemma \ref{val-ne}, Theorem \ref{e-zeta-thm1}(4), Theorem \ref{e-zeta-thm1}(5) and (\ref{note}) we have
\begin{equation}\label{zeta-neg-val}
E_{N,k}(x;\bar\omega)= m^k \sum_{0\leq j_i<m}(-1)^{|\bar j|}E_{N,k}\left(\frac{x+\bar j\cdot\bar\o}{m};\bar\omega\right),
\end{equation}
where the sum is over all vectors $\bar j=(j_1,\ldots,j_N)$ with $0\leq j_i<m.$
From Theorem \ref{e-zeta-thm1}(4) it is easy to see that
\begin{equation}\label{zeta-g-4}
\begin{aligned}
\frac{\frac{x+\bar j\cdot\bar\o}{m}}{\left\la \frac{x+\bar j\cdot\bar\o}{m}\right\ra}
&\frac{\partial}{\partial s}\left(\frac1{s-1}\zeta_{p,E,N}(s,x;\bar\omega)\right)\biggl|_{s=0} \\
&=\frac{\frac{x+\bar j\cdot\bar\o}{m}}{\left\la \frac{x+\bar j\cdot\bar\o}{m}\right\ra}
\frac{\partial}{\partial s}\left(\frac{\la m\ra^{1-s}}{s-1}
\sum_{0\leq j_i<m}^{m-1}(-1)^{|\bar j|}\zeta_{p,E,N}\left(s,\frac{x+\bar j\cdot\bar\o}{m};\bar\omega\right)\right)\biggl|_{s=0}.
\end{aligned}
\end{equation}
In view of (\ref{mul-e-log}) and (\ref{note}), the left side of (\ref{zeta-g-4}) becomes
\begin{equation}\label{left-log}
\frac{\la m\ra}{m}\frac{x}{\la x\ra}\frac{\partial}{\partial s}\left(\frac1{s-1}\zeta_{p,E,N}(s,x;\bar\omega)\right)\biggl|_{s=0}
=\frac{\la m\ra}m\Log(x;\bar\omega).
\end{equation}
Since
$\zeta_{p,E,N}(0,x;\bar\o)=\frac{\la x\ra}{x}E_{N,1}(x;\bar\o),$
by (\ref{mul-e-log}), (\ref{note}), (\ref{zeta-neg-val}) with $k=1,$ (\ref{zeta-g-4}) and (\ref{left-log}),  we get
$$
\begin{aligned}
\Log(x;\bar\omega)
&=\frac{\frac{x+\bar j\cdot\bar\o}{m}}{\left\la \frac{x+\bar j\cdot\bar\o}{m}\right\ra} m\log_p m
\sum_{0\leq j_i<m}(-1)^{|\bar j|}\zeta_{p,E,N}\left(0,\frac{x+\bar j\cdot\bar\o}{m};\bar\omega\right)
 \\
&\quad+\la m\ra \sum_{0\leq j_i<m}(-1)^{|\bar j|} \Log\left(\frac{x+\bar j\cdot\bar\o}{m};\bar\omega\right)  \\
&= m\log_p m \sum_{0\leq j_i<m}(-1)^{|\bar j|} E_{N,1}\left(\frac{x+\bar j\cdot\bar\o}{m};\bar\omega\right)
 \\
 &\quad+m \sum_{0\leq j_i<m}(-1)^{|\bar j|} \Log\left(\frac{x+\bar j\cdot\bar\o}{m};\bar\omega\right)  \\
 &= E_{N,1}(x;\bar\o) \log_p m \\
  &\quad+m \sum_{0\leq j_i<m}(-1)^{|\bar j|} \Log\left(\frac{x+\bar j\cdot\bar\o}{m};\bar\omega\right),
\end{aligned}
$$
which proves (4).

(5) Observe that for given $x$ and $\bar\o,$ we have
$$\frac{\partial}{\partial x}[(x+\bar\o\cdot\bar t\,)(\log_p(x+\bar\o\cdot\bar t\,)-1)]=\log_p(x+\bar\o\cdot\bar t\,).$$
Thus, by Proposition \ref{mul-e-log-int}, we have
$\psi_{p,E,N}(x;\bar\omega)=\int_{\mathbb Z_p^N} \log_p(x+\bar\o\cdot\bar t\,) d\mu_{-1}(\bar t\,)$
and the first equality of (5) follows by induction. Moreover for $|x|_p>\|\bar\o\|_p$ we have
\begin{equation}\label{int-<>}
\la x+\bar\o\cdot\bar t\,\ra=\frac{\la x\ra}{x} (x+\bar\o\cdot\bar t\,)\quad\textrm{for~~all}~~\bar t\in\mathbb Z_p^N
\end{equation}
and from the definition of $\zeta_{p,E,N}(s,x;\bar\omega)$ (\ref{mul-def-E-zeta}) and the first part of (5), we have
$$\begin{aligned}
\psi_{p,E,N}^{(k+1)}(x;\bar\omega)&=(-1)^{k+1}\int_{\mathbb Z_p^N} \frac1{(x+\bar\o\cdot\bar t\,)^k}d\mu_{-1}(\bar t\,) \\
&=(-1)^{k+1}\left(\frac {\la x\ra}x\right)^{k}\zeta_{p,E,N}(k+1,x;\bar\omega),
\end{aligned}$$
which is the second equality of (5).
\end{proof}

Now we give computationally efficient formulas for our functions in the case that the argument $x$ has $p$-adic absolute value larger than the norm
of $\bar\o.$

\begin{theorem}\label{t-expan-p=ga}
Suppose $\o_1,\ldots,\o_N,x\in\mathbb C_p^\times$ and $|x|_p>\|\bar\o\|_p.$ Then there is an identity of analytic functions
$$\zeta_{p,E,N}(s,x;\bar\o)=\la x\ra^{1-s}\sum_{j=0}^\infty\binom{1-s}{j}E_{N,j}(0;\bar\o)x^{-j}$$
on a disc of positive radius about $s=0.$
If in addition $\o_1,\ldots,\o_N,x$ are so chosen to lie in a finite extension $K$ of $\mathbb Q_p$ whose ramification index over $\mathbb Q_p$
is less than $p-1,$ then this formula is valid for $s\in\mathbb C_p$ such that $|s|_p<|\pi|_p^{-1}p^{-1/(p-1)},$ where $(\pi)$ is the maximal ideal
of the ring of integers $O_K$ of $K.$
\end{theorem}
\begin{proof} Under the stated hypotheses, for all $\bar t\in\mathbb Z_p^N,$ by (\ref{int-<>}), we have
$$\begin{aligned}
\la x+\bar\o\cdot\bar t\,\ra^{1-s}&=\left(\frac {\la x\ra}x (x+\bar\o\cdot\bar t\,)\right)^{1-s}  \\
&=\la x\ra^{1-s}\sum_{j=0}^\infty\binom{1-s}{j}(\bar\o \cdot\bar t\,)^j x^{-j}
\end{aligned}$$
as an identity of analytic functions.
From (\ref{eu-N=0}), denote by $$(\bar\o \cdot\bar t\,)^j=E_{0,j}(\bar\o \cdot\bar t;-).$$ If integrating  the above equality
with respect to $t_1,\ldots, t_N$, then by (\ref{mul-def-E-zeta}) and Corollary \ref{co-lem-1}, we have
$$\begin{aligned}
\zeta_{p,E,N}(s,x;\bar\o)&=\int_{\mathbb Z_p^N}\la x+\bar\o\cdot\bar t\,\ra^{1-s}d\mu_{-1}(\bar t\,) \\
&=\la x\ra^{1-s}\sum_{j=0}^\infty\binom{1-s}{j} x^{-j}\int_{\mathbb Z_p^N}(\bar\o \cdot\bar t\,)^jd\mu_{-1}(\bar t\,) \\
&=\la x\ra^{1-s}\sum_{j=0}^\infty\binom{1-s}{j}E_{N,j}(0;\bar\o)x^{-j}.
\end{aligned}$$
Therefore the assertion is clear.
\end{proof}

\begin{theorem}[Stirling's series]\label{Stirling}
Suppose $\o_1,\ldots,\o_N,x\in\mathbb C_p^\times$ and $|x|_p>\|\bar\o\|_p.$ Then
$$\begin{aligned}
\Log(x;\bar\omega)&=x(\log_p x-1)+E_{N,1}(0;\bar\o)\log_p x \\
&\quad+\sum_{j=2}^\infty\frac{(-1)^{j}}{j(j-1)}E_{N,j}(0;\bar\o)x^{-j+1}.
\end{aligned}$$
\end{theorem}
\begin{proof}
We have the following expansions (see \cite[p.~38]{Fo}):
$$\langle x\rangle^{1-s}=\langle x\rangle(1-s\log_p\langle x\rangle+\cdots)$$
and
$$\binom{1-s}j=\frac{(-1)^{j+1}}{j(j-1)}s+\cdots,$$
provided $j\geq2.$
Rewrite the expansion of Theorem \ref{t-expan-p=ga} on a disc of positive radius about $s=0$ as
\begin{equation}\label{ga-pf-1}
\begin{aligned}
\la x\ra^{1-s}\sum_{j=0}^\infty\binom{1-s}{j}E_{N,j}(0;\bar\o)x^{-j}&=\la x\ra^{1-s}\biggl(1+(1-s)x^{-1}E_{N,1}(0;\bar\o) \\
&\quad+\sum_{j=2}^\infty\left(\frac{(-1)^{j+1}}{j(j-1)}s+\cdots\right)x^{-j}E_{N,j}(0;\bar\o)\biggl).
\end{aligned}
\end{equation}
By the definition (\ref{mul-e-log}), (\ref{ga-pf-1}) and Leibniz's rule, we obtain
$$\begin{aligned}
\Log(x;\bar\omega)&=\frac{x}{\la x\ra}\frac{\partial}{\partial s}\left(\frac1{s-1}\zeta_{p,E,N}(s,x;\bar\omega)\right)\biggl|_{s=0} \\
&=\frac{x}{\la x\ra}\frac{\partial}{\partial s}\biggl[
\frac{\la x\ra^{1-s}}{s-1}\biggl(1+(1-s)x^{-1}E_{N,1}(0;\bar\o)  \\
&\quad+\sum_{j=2}^\infty\left(\frac{(-1)^{j+1}}{j(j-1)}s+\cdots\right)x^{-j}E_{N,j}(0;\bar\o)\biggl)\biggl]\biggl|_{s=0} \\
&=(x\log_px-x)(1+x^{-1}E_{N,1}(0;\bar\o)) \\
&\quad+E_{N,1}(0;\bar\o)+\sum_{j=2}^\infty\frac{(-1)^{j}}{j(j-1)}x^{-j+1}E_{N,j}(0;\bar\o).
\end{aligned}$$
This completes the proof.
\end{proof}

\begin{remark}
1. Since ${\rm Log}\, \Gamma_{\! D,E,1}(x;1)={\rm Log}\, \Gamma_{\! D,E}(x)$ for $x\in \mathbb C_p\setminus\mathbb Z_p$
and $E_{1,j}(0;1)=E_j(0)$ for $n\in\mathbb N_0,$ when $N=1$ and $\omega_1=1$ Theorem \ref{Stirling} becomes the known relation \cite[Theorem 2.3]{KS13}
$${\rm Log}\, \Gamma_{\! D,E}(x)=x(\log_p x-1)+E_{1}(0)\log_p x
+\sum_{j=2}^\infty\frac{(-1)^{j}}{j(j-1)}E_{j}(0)x^{-j+1}.$$
If $j$ is even, then we have $E_j(0)=0,$ and we may replace the symbol $(-1)^{j}$ by $-1$ in the above formula.

2. From Theorem \ref{Stirling}, Stirling's series expansion for $\Log(x;\bar\omega)$ agrees exactly with the asymptotic expansion
for $\log \Gamma_{E,N}(x;\bar\omega)$ given by Theorem \ref{int-re-3} where the infinite series converges $p$-adically for large $x$ and
the error term $R_{N,M}$ vanishes.
This phenomenon was first observed   by Diamond \cite[Theorem 6]{Di} for $N=1$, and  by Tangedal and Young for  general positive integers $N$ \cite[Theorem 4.2]{TY}.
\end{remark}

By differentiating both sides of Theorem \ref{Stirling} with respect to $x,$
we have the following Laurent series expansion of $\psi_{p,E,N}^{(k+1)}(x;\bar\omega).$

\begin{corollary}\label{h-Di-ft-Eu}
Let $\o_1,\ldots,\o_N,x\in\mathbb C_p^\times$ and $|x|_p>\|\bar\o\|_p.$  Then
$$\psi_{p,E,N}(x;\bar\omega)=\log_p x+E_{N,1}(0;\bar\o)\frac1x+\sum_{j=1}^{\infty}\frac{(-1)^j}{j+1}\frac{1}{x^{j+1}}E_{N,j+1}(0;\bar\omega)$$
and
$$\psi_{p,E,N}^{(k+1)}(x;\bar\omega)=(k-1)!\sum_{j=0}^{\infty}(-1)^j\binom{-j-1}{k-1}\frac{1}{x^{k+j}}E_{N,j}(0;\bar\omega),$$
where $k\in\mathbb N.$
\end{corollary}

Although for $x\in\Lambda$ our functions $\zeta_{p,E,N}(s,x;\bar\omega)$ and $\Log(x;\bar\omega)$ are undefined,
we  may extend their  definitions to $x\in\Lambda$ and $\|\bar\omega\|_p<1$ as zero values, and denote them by $\zeta_{p,E,N}^*(s,x;\bar\omega)$ and $L\Gamma_{p,E,N}(x;\bar\omega)$ (see Kashio \cite{Kashio}, Tangedal and Young \cite{TY}).
The remaining case, when $\|\bar\omega\|_p\geq1$ and $x\in\Lambda$,
can be expressed in terms of the above definitions, as we show in what follows.

We consider the function
\begin{equation}\label{sta-zet}
\zeta_{p,E,N}^*(s,x;\bar\omega)=\int_{\mathbb Z_p^N}f^*(x,\bar t\,)d\mu_{-1}(\bar t\,),
\end{equation}
where
$$f^*(x,\bar t\,)=\begin{cases}f(x,\bar t\,) &\text{if } |x+\bar\omega\cdot\bar t |_p=1, \\
0 &\text{if } |x+\bar\omega\cdot\bar t |_p<1
\end{cases}$$
and $f(x,\bar t\,)$ is as defined in (\ref{def-ft}). Then we define
\begin{equation}\label{mul-e-log-star}
L\Gamma_{p,E,N}(x;\bar\omega)=\frac{x}{\la x\ra}\frac{\partial}{\partial s}\left(\frac1{s-1}\zeta_{p,E,N}^*(s,x;\bar\omega)\right)\biggl|_{s=0}.
\end{equation}

\begin{theorem}
Suppose that $\|\bar\omega\|_p\geq1$ and $x\in\Lambda.$ Then we have
$$\zeta_{p,E,N}^*(s,x;\bar\omega)=\sum_{{\substack{0\leq j_i<p \\ |x+\bar j\cdot\bar\omega|_p=\|\bar\omega\|_p}}}
(-1)^{|\bar j|}\zeta_{p,E,N}\left(s,\frac{x+\bar j\cdot\bar\omega}{p};\bar\omega\right)$$
and
$$L\Gamma_{p,E,N}(x;\bar\omega)=\sum_{{\substack{0\leq j_i<p \\ |x+\bar j\cdot\bar\omega|_p=\|\bar\omega\|_p}}}
(-1)^{|\bar j|}\Log\left(\frac{x+\bar j\cdot\bar\omega}{p};\bar\omega\right),$$
where the sums are all over all vectors $\bar j=(j_1,\ldots,j_N)$ with $0\leq j_i<p ~(i=1,\ldots,N)$ and $|x+\bar j\cdot\bar\omega|_p=\|\bar\omega\|_p.$
\end{theorem}
\begin{proof}
Let us assume that $\|\bar\omega\|_p=1$ and $x\in\Lambda.$ From (\ref{mul-ft-4}) and (\ref{sta-zet}), we obtain
\begin{equation}\label{sta-zet-pf}
\begin{aligned}
\zeta_{p,E,N}^*(s,x;\bar\omega)&=\int_{\mathbb Z_p^N}f^*(x,\bar t\,)d\mu_{-1}(\bar t\,) \\
&=\sum_{0\leq j_i<p }
(-1)^{|\bar j|}\int_{\mathbb Z_p^N}f^*(x,\bar j+p\bar t\,)d\mu_{-1}(\bar t\,)  \\
&=\sum_{0\leq j_i<p}
(-1)^{|\bar j|}\int_{\mathbb Z_p^N}f^*(x+\bar j\cdot\bar\omega,p\bar t\,)d\mu_{-1}(\bar t\,)  \\
&=\sum_{{\substack{0\leq j_i<p  \\ |x+\bar\omega\cdot\bar t |_p=1}}}
(-1)^{|\bar j|}\int_{\mathbb Z_p^N}f(x+\bar j\cdot\bar\omega,p\bar t\,)d\mu_{-1}(\bar t\,)  \\
&=\sum_{{\substack{0\leq j_i<p  \\ |x+\bar\omega\cdot\bar t |_p=1}}}
(-1)^{|\bar j|}\int_{\mathbb Z_p^N}f\left(\frac{x+\bar j\cdot\bar\omega}{p},\bar t\right)d\mu_{-1}(\bar t\,)  \\
&=\sum_{{\substack{0\leq j_i<p  \\ |x+\bar\omega\cdot\bar t |_p=1}}}
(-1)^{|\bar j|}\zeta_{p,E,N} \left(s,\frac{x+\bar j\cdot\bar\omega}{p};\bar \omega\right),
\end{aligned}
\end{equation}
which proves the first part in the case $\|\bar\omega\|_p=1.$ The case of general $\|\bar\omega\|_p\geq1$ may be deduced from the dilation relation
(Theorem \ref{e-zeta-thm1}(2) above and \cite[Eq.~(5.7)]{Kashio}) of $\zeta_{p,E,N}$ and $\zeta_{p,E,N}^*.$
The second part follows from (\ref{mul-e-log-star}) and the first part.
\end{proof}

\end{document}